\documentclass[11pt]{article}

\usepackage{amsmath,amssymb,amsthm}
\usepackage{tikz}
\usepackage{float}
\usepackage[margin=1in]{geometry}

\newtheorem{lemma}{Lemma}[section]
\newtheorem{corollary}{Corollary}[section]
\newtheorem{theorem}{Theorem}[section]

\title{Every Minimal Counterexample to the Erd\H{o}s--Gy\'arf\'as Conjecture is Predominantly Cubic}

\author{
Avery Carr\\
Independent Researcher\\
\texttt{avery.carr@ymail.com}
}

\date{Updated: May 13, 2026}

\begin{document}

\maketitle

\begin{abstract}
A minimal counterexample to the Erd\H{o}s--Gy\'arf\'as conjecture is a graph of minimum possible order and size with minimum degree at least \(3\) that contains no cycle whose length is a power of \(2\). Markstr\"om observed that any such graph must contain an independent set of vertices of degree at least \(4\) together with a nonempty set of vertices of degree exactly \(3\). As an immediate consequence, every regular minimal counterexample must be cubic. Building on this structure, two additional consequences are derived. First, every vertex of a minimal counterexample is adjacent to a vertex of degree exactly \(3\). Second, at least \(4/7\) of the vertices of any minimal counterexample must have degree exactly \(3\). 
\end{abstract}

\bigskip


\medskip


\section*{Notation}

All graphs considered in this note are finite, simple, and undirected. Let \(G=(V(G),E(G))\) be a graph with vertex set \(V(G)\) and edge set \(E(G)\). For a vertex \(v\in V(G)\), the neighborhood of \(v\) is
\[
N(v)=\{u\in V(G):uv\in E(G)\},
\]
and the degree of \(v\) is
\[
d(v)=|N(v)|.
\]

The minimum degree of \(G\) is denoted by
\[
\delta(G)=\min\{d(v):v\in V(G)\}.
\]

A graph \(H\) is a proper subgraph of \(G\) if \(H\subsetneq G\). A graph \(G\) is called \(k\)-regular if every vertex of \(G\) has degree \(k\). A cubic graph is a \(3\)-regular graph.

In this note, a minimal counterexample to the Erd\H{o}s--Gy\'arf\'as conjecture means a graph \(G\) with \(\delta(G)\geq 3\) that contains no cycle whose length is a power of \(2\), chosen with minimum possible order and, subject to that, minimum possible size.

\section*{Introduction}

The Erd\H{o}s--Gy\'arf\'as conjecture asks whether every graph \(G\) with minimum degree \(\delta(G)\geq 3\) contains a cycle whose length is a power of two \cite{erdos1997}. Despite its simple formulation, the conjecture remains open in general and has only been verified for restricted graph classes, including planar graphs, cubic claw-free graphs, \(3\)-connected cubic planar graphs, $P_8$-free graphs, and $P_{10}$-free graphs \cite{DanielShauger2001,NowbandeganiEtAl2014,HeckmanKrakovski2013,GaoShan2022, hushen}.

Computational work of Royle and Markstr\"om further suggests that any counterexample must be highly constrained. In particular, their investigations imply that any cubic counterexample must contain at least \(30\) vertices, while extremal constructions show that the smallest power-of-two cycle lengths can first occur at comparatively large values such as \(16\) \cite{Markstrom2004}.

Recent work of the author establishes the conjecture for graphs of diameter \(2\), proving that every graph \(G\) with \(\operatorname{diam}(G)=2\) and \(\delta(G)\geq 3\) contains a cycle of length \(4\) or \(8\) \cite{CarrDiameter2}. Motivated by the increasingly rigid structure expected of minimal counterexamples, this note studies structural restrictions on such graphs.

\section*{Main Result}

\begin{lemma}
Let \(G\) be a minimal counterexample to the Erd\H{o}s--Gy\'arf\'as conjecture. Then \(\delta(H)\leq 2\) for every proper subgraph \(H\subsetneq G\).
\end{lemma}

\begin{proof}
Let \(G\) be a minimal counterexample with respect to order and size, and suppose \(H\subsetneq G\) is a proper subgraph with \(\delta(H)\geq 3\). By minimality of \(G\), the graph \(H\) cannot be a counterexample. Hence, \(H\) contains a cycle whose length is a power of \(2\). Since \(H\) is a subgraph of \(G\), the same cycle occurs in \(G\), contradicting that \(G\) is a counterexample. Therefore, \(\delta(H)\leq 2\) for every proper subgraph \(H\subsetneq G\).
\end{proof}

The following corollaries and theorem strengthen the structural picture first
described by Markstr\"om~\cite{Markstrom2004}. In particular,
Corollary~0.1(1) shows that the cubic vertices form a dominating set in every
minimal counterexample. Corollary~0.1(2), originally observed by
Markstr\"om~\cite{Markstrom2004}, states that the vertices of degree at least $4$
form an independent set. Theorem~0.1 then uses this structural restriction to
establish an explicit lower bound on the proportion of cubic vertices.

\begin{corollary}
Let \(G\) be a minimal counterexample to the Erd\H{o}s--Gy\'arf\'as conjecture.

\begin{enumerate}
    \item Every vertex of \(G\) is adjacent to a vertex of degree exactly \(3\).

    \item The set of vertices of degree at least \(4\) forms an independent set.
\end{enumerate}
\end{corollary}

\begin{proof}

(1) Let \(v\in V(G)\). Since \(G-v\) is a proper subgraph of \(G\), Lemma \(0.1\) gives \(\delta(G-v)\leq 2\). Since every vertex of \(G\) has degree at least \(3\), the only way a vertex in \(G-v\) can have degree at most \(2\) after deleting \(v\) is if it was adjacent to \(v\) and had degree exactly \(3\) in \(G\). Therefore, \(v\) is adjacent to a vertex of degree exactly \(3\).

\vspace{1em}

(2) Suppose \(u,v\in V(G)\) are adjacent vertices with \(d(u)\geq 4\) and \(d(v)\geq 4\). Deleting the edge \(uv\) decreases both degrees by exactly \(1\), so both vertices retain degree at least \(3\), while all other vertex degrees remain unchanged. Hence \(\delta(G-uv)\geq 3\), contradicting Lemma \(0.1\). Therefore, no two vertices of degree at least \(4\) can be adjacent.

\end{proof}

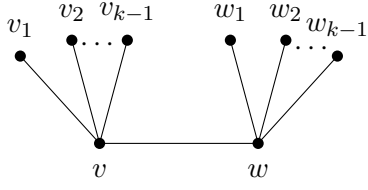
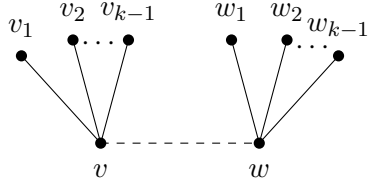
\begin{figure}[H]
\centering
\begin{tikzpicture}[scale=1.05]

\fill (0,0) circle (2pt);
\fill (2,0) circle (2pt);

\fill (-1,1.1) circle (2pt);
\fill (-0.35,1.3) circle (2pt);
\fill (0.35,1.3) circle (2pt);

\fill (1.65,1.3) circle (2pt);
\fill (2.35,1.3) circle (2pt);
\fill (3,1.1) circle (2pt);

\node at (0,-0.35) {$v$};
\node at (2,-0.35) {$w$};

\node at (-1,1.45) {$v_1$};
\node at (-0.35,1.65) {$v_2$};
\node at (0.35,1.65) {$v_{k-1}$};

\node at (1.65,1.65) {$w_1$};
\node at (2.35,1.65) {$w_2$};
\node at (3,1.45) {$w_{k-1}$};

\node at (0,1.25) {$\cdots$};
\node at (2.7,1.20) {$\cdots$};

\draw (0,0)--(2,0);

\draw (0,0)--(-1,1.1);
\draw (0,0)--(-0.35,1.3);
\draw (0,0)--(0.35,1.3);

\draw (2,0)--(1.65,1.3);
\draw (2,0)--(2.35,1.3);
\draw (2,0)--(3,1.1);

\end{tikzpicture}

\caption{An edge \(vw\) in a \(k\)-regular graph, where \(k\geq 4\). Each endpoint has \(k-1\geq 3\) other incident edges.}
\end{figure}

\begin{figure}[H]
\centering
\begin{tikzpicture}[scale=1.05]

\fill (0,0) circle (2pt);
\fill (2,0) circle (2pt);

\fill (-1,1.1) circle (2pt);
\fill (-0.35,1.3) circle (2pt);
\fill (0.35,1.3) circle (2pt);

\fill (1.65,1.3) circle (2pt);
\fill (2.35,1.3) circle (2pt);
\fill (3,1.1) circle (2pt);

\node at (0,-0.35) {$v$};
\node at (2,-0.35) {$w$};

\node at (-1,1.45) {$v_1$};
\node at (-0.35,1.65) {$v_2$};
\node at (0.35,1.65) {$v_{k-1}$};

\node at (1.65,1.65) {$w_1$};
\node at (2.35,1.65) {$w_2$};
\node at (3,1.45) {$w_{k-1}$};

\node at (0,1.25) {$\cdots$};
\node at (2.7,1.20) {$\cdots$};

\draw[dashed] (0,0)--(2,0);

\draw (0,0)--(-1,1.1);
\draw (0,0)--(-0.35,1.3);
\draw (0,0)--(0.35,1.3);

\draw (2,0)--(1.65,1.3);
\draw (2,0)--(2.35,1.3);
\draw (2,0)--(3,1.1);

\node[draw=none] at (1,-1.05) {$d_{G-vw}(v)=d_{G-vw}(w)=k-1\geq 3$};

\end{tikzpicture}

\caption{After deleting the edge \(vw\), the vertices \(v\) and \(w\) still have degree at least \(3\), while all other vertices retain their original degrees. Hence \(\delta(G-vw)\geq 3\).}
\end{figure}

\begin{corollary}
If \(G\) is a regular minimal counterexample to the Erd\H{o}s--Gy\'arf\'as conjecture, then \(G\) is cubic.
\end{corollary}

\begin{proof}
Suppose \(G\) is \(k\)-regular. Since \(\delta(G)\geq3\), one has \(k\geq3\). If \(k\geq4\), then every vertex has degree at least \(4\), contradicting Corollary \(0.1(2)\). Therefore \(k=3\), and hence \(G\) is cubic.
\end{proof}
\begin{theorem}
Let \(G\) be a minimal counterexample to the Erd\H{o}s--Gy\'arf\'as conjecture. Then at least \(4/7\) of the vertices of \(G\) have degree exactly \(3\).
\end{theorem}

\begin{proof}
Let
\[
V_3=\{v\in V(G):d(v)=3\}
\]
and
\[
V_{\geq4}=\{v\in V(G):d(v)\geq4\}.
\]

By Corollary \(0.1(2)\), the set \(V_{\geq4}\) is independent. Hence every edge incident to a vertex of \(V_{\geq4}\) joins it to a vertex of \(V_3\).

Let \(e(V_3,V_{\geq4})\) denote the number of edges between the two sets. Since every vertex in \(V_{\geq4}\) has degree at least \(4\), one has
\[
e(V_3,V_{\geq4})\geq4|V_{\geq4}|.
\]
On the other hand, every vertex in \(V_3\) has degree exactly \(3\), so
\[
e(V_3,V_{\geq4})\leq3|V_3|.
\]
Therefore,
\[
4|V_{\geq4}|\leq3|V_3|.
\]

Since \(V(G)=V_3\cup V_{\geq4}\), it follows that
\[
|V(G)|
=
|V_3|+|V_{\geq4}|
\leq
|V_3|+\frac34|V_3|
=
\frac74|V_3|.
\]
Hence,
\[
|V_3|\geq\frac47|V(G)|.
\]

Thus at least \(4/7\) of the vertices of \(G\) have degree exactly \(3\).
\end{proof}

\section*{Acknowledgments}

The author would like to thank Klas Markstr\"om for prior structural observations related to minimal counterexamples to the Erd\H{o}s--Gy\'arf\'as conjecture, as well as the editors for their time and consideration.

\end{document}